\documentclass[a4paper]{scrartcl}

\usepackage{fullpage}

\usepackage{lmodern}

\usepackage[LGR,T1]{fontenc}
\usepackage{amsthm}

\usepackage{amsmath}

\usepackage{amsfonts}
\usepackage{amssymb}

\usepackage{array}
\usepackage{multirow}

\usepackage[english]{babel}

\usepackage[usenames,dvipsnames]{xcolor}
\usepackage{url}

\usepackage{aliascnt}

\newtheorem{theorem}{Theorem}
\theoremstyle{plain}

\newaliascnt{corollary}{theorem}  
\newtheorem{corollary}[corollary]{Corollary}
\aliascntresetthe{corollary}

\newaliascnt{lemma}{theorem}  
\newtheorem{lemma}[lemma]{Lemma}  
\aliascntresetthe{lemma}

\newaliascnt{proposition}{theorem}  
\newtheorem{proposition}[proposition]{Proposition}  
\aliascntresetthe{proposition}

\newtheorem{result}[theorem]{Result}
\newtheorem*{result*}{Result}

\theoremstyle{definition}

\newtheorem{definition}[theorem]{Definition}

\theoremstyle{remark}

\newtheorem*{claim*}{Claim}

\newtheorem{remark}[theorem]{Remark}

\AtBeginDocument{}

\newcommand{\abs}[1]{\left|#1\right|}

\newcommand{\reals}{\ensuremath{\mathbb{R}}}

\newcommand{\integers}{\ensuremath{\mathbb{Z}}}

\newcommand{\complexes}{\ensuremath{\mathbb{C}}}
\newcommand{\torus}{\ensuremath{\mathbb{T}}}

\AtBeginDocument{
	
	\renewcommand{\geq}{\geqslant}
}

\usepackage{etoolbox}

\usepackage[eprint=false,url=false,doi=false,mincrossrefs=1,backend=bibtex]{biblatex}
\usepackage{hyperref}

\title{Invariant Meromorphic Functions on the Wallpaper Groups}
\author{Richard Chapling\thanks{\ Trinity College, Cambridge}}

\begin{document}

\maketitle

\begin{abstract}
  We construct the meromorphic functions invariant under the action of the sense-preserving wallpaper groups on the complex plane. We discuss possible generalisations of this to the general wallpaper groups. This provides the answer to a question posed on StackExchange \cite{36737} regarding the possible symmetries of elliptic functions.
\end{abstract}

\section{Introduction}
\label{sec:introduction}

Complex-valued functions invariant under the action of wallpaper groups have been considered by Farris and coauthors, in a number of articles \cite{Farris:1998bh,Farris:2002dq} and recently, a semi-popular book \cite{Farris:2015qf}. The functions considered in these studies are essentially all solutions to the wave equation on the Euclidean plane \( \reals^2 \).

However, since we may also take the plane to be the complex plane \( \complexes \), a natural alternative to this idea is to consider \emph{meromorphic} functions which are invariant under the action of the wallpaper groups. It is clear that they shall all be elliptic (see \autoref{thm:wallell} below), and hence we can expect to classify them completely using the decomposition in terms of Weierstrass elliptic functions. This is the aim of this paper. 

The first investigations in this area appear to be those posted as answers to two questions \cite{35671,36737} posed on \texttt{math.stackexchange}, where a number of examples of elliptic functions invariant under particular wallpaper symmetries are given (in particular, those corresponding to the wallpaper groups containing only sense-preserving transformations). We shall verify the conclusions therein, and provide what appears to be the natural extension to the other wallpaper groups, which include reflexions, by specifying that reflexions \(R\) in the group act by conjugation in both the domain and the range: this acts to keep the function meromorphic, while still allowing for extra structure to manifest itself.

In the remainder of this section, we outline the facts that we require from the theory of elliptic functions, and of the wallpaper groups. In \autoref{sec:meromorphic-wallpaper-funct}, we consider the functions invariant under the groups which do not contain reflexions, and then in \autoref{sec:gener-refl-groups}, we consider a way of generalising the idea to the rest of the wallpaper groups. In the last section, we suggest a number of ways in which the investigation could be extended to other groups and geometries.

%Since both wallpaper groups and meromorphic functions act on the complex plane, it is natural to ask what the meromorphic \(G\)-invariant functions of a given wallpaper group \(G\) are. 

\subsection{Elliptic functions}
\label{sec:elliptic-functions}

The facts in this section may be found in any standard text which discusses elliptic functions.\footnote{For example, \cite{Jones:1987ly}, Ch.~3 or \cite{WhittakerWatson:1927ve}, Ch.~XX.}

\begin{definition}
  An \emph{elliptic function} is a meromorphic function \( \complexes \to \mathbb{CP}^1 \) with two nonparallel periods: that is, there are nonzero complex \(\omega_1 \) and \( \omega_2 \), \( \omega_1/\omega_2 \notin \reals \), with \( f(z)=f(z+\omega_1)=f(z+\omega_2)\). The set \( \Lambda = \integers \omega_1 + \integers \omega_2\) is called the \emph{period lattice} of \(f\).
\end{definition}

We recall that the Weierstrass elliptic function associated with the lattice \(\Lambda\) is
\begin{equation}
  \label{eq:1}
  \wp{(z;\Lambda)} := \frac{1}{z^2}  + \sideset{}{'}\sum_{\omega \in \Lambda} \frac{1}{(z-\omega)^2} - \frac{1}{\omega^2},
\end{equation}
where \(\sum' \) conventionally means that the term \(\omega=0\) is omitted.

This has the following basic properties:
\begin{enumerate}
\item \( \wp \) is even.
\item It has a pole of order two at each point of \( \Lambda \), and is elsewhere analytic.
\item It satisfies the differential equation
  \begin{equation}
    \label{eq:2}
    \wp'^2 = 4\wp^3 - g_2 \wp - g_3,
  \end{equation}
  where \(g_2,g_3\) are the Eisenstein series \(g_2=60\sum' \omega^{-4}\) and \( g_3 = 140\sum' \omega^{-6}\).
\item Thus, it is the inverse of the integral
  \begin{equation}
    \label{eq:3}
    u = \int_{\wp}^{\infty} \frac{ds}{\sqrt{4s^3-g_2 s-g_3}}.
  \end{equation}
\item Most importantly for our purposes, any elliptic function with period lattice \(\Lambda\) can be written uniquely in the form
  \begin{equation}
    \label{eq:4}
    f(z) = R(\wp(z))+S(\wp(z))\wp'(z),
  \end{equation}
  where \(R,S\) are rational functions.
\item Later we shall have reason to use the addition formula for the Weierstrass elliptic function in the form
  \begin{equation}
    \label{eq:38}
    \wp(z+w) = \frac{1}{4}\left( \frac{\wp'(z)-\wp'(w)}{\wp(z)-\wp(w)} \right)^2 -\wp(z)-\wp(w)
  \end{equation}
\end{enumerate}

\subsection{Wallpaper groups}
\label{sec:wallpaper-groups}

\begin{definition}
  A \emph{wallpaper group} is a group of rigid motions of the plane, containing two non-parallel translations.
\end{definition}

It is well-known\footnote{One reference is \cite{Coxeter:1980ve}, which also contains presentations for each of the wallpaper groups; these differ slightly from ours primarily because the action of complex transformations on elliptic functions is most efficiently chosen to have as few reflexions as possible.} that there are precisely 17 wallpaper groups (up to isomorphism, of course).

The wallpaper groups induce an obvious group action on the complex plane, via:
\begin{equation}
  \label{eq:5}
  g.z = az+b
\end{equation}
if \(g\) is sense-preserving, and
\begin{equation}
  \label{eq:6}
  g.z = a\bar{z}+b,
\end{equation}
if \( g\) is sense-reversing (i.e. decomposes to an odd number of reflexions).

Since the conjugate would force the transformed function to be antimeromorphic, restricting to meromorphic functions requires that the group can contain no reflexions or glide reflexions: in particular, this excludes pg, pm, cm, pgg, pmg, pmm, cmm, p3m1, p31m, p4g, p4m and p6m. We are left with the 5 \emph{sensible groups},\footnote{so-called since they preserve \emph{sense}, or orientation.} p1, p2, p4, p3 and p6.

\section{Meromorphic wallpaper functions}
\label{sec:meromorphic-wallpaper-funct}

Let \(  G \) be a wallpaper group. We call \(f\) a \emph{wallpaper function} if for every complex number \(z\), \( f(g.z)=f(z) \) for every \(g \in G\): i.e. \(f\) is \( G\)-invariant.

As remarked above, considering meromorphic wallpaper functions forces us to restrict to the sensible wallpaper groups. We have the following immediate

\begin{result}
\label{thm:wallell}
  Every meromorphic wallpaper function is elliptic.
\end{result}
\begin{proof}
  \(G\) contains two nonparallel translations, which may be represented by \( z \mapsto z+ \omega_1 \), \( z \mapsto z + \omega_2 \), with \( \omega_1/\omega_2 \notin \reals\). Since \(f\) is \(G\)-invariant, we have
  \begin{equation}
    \label{eq:7}
    f(z) = f(z+\omega_1)=f(z+\omega_2),
  \end{equation}
  which is exactly the condition that \( f\) is elliptic.
\end{proof}

\begin{corollary}
  Every meromorphic wallpaper function is of the form
  \begin{equation}
    \label{eq:8}
    f(z) = R(\wp(z)) + S(\wp(z)) \wp'(z),
  \end{equation}
  where \(R,S\) are rational functions.
\end{corollary}

It is fairly clear how we must proceed now: we shall determine the restrictions on \(R,S\) for each of the sensible wallpaper groups. Since each of these is generated by the translations and a single rotation, it is sufficient to examine the transformation of \eqref{eq:8} under the rotation.

\subsection{p1}
\label{sec:p1}

The lattice cell for the group p1 is in general a parallelogram. This group possesses no isometries extra to the translations, and so no further restriction is imposed: we conclude that \emph{any} elliptic function is p1-invariant.

\subsection{p2}
\label{sec:p2}

Here again the translation lattice cell is in general a parallelogram. Now, however, we have also a rotation by \(\pi\), i.e. an idempotent map \( r \), \(r^2 = 1 \) that satisfies
\begin{equation}
  \label{eq:9}
  ar=a^{-1}, \quad br=b^{-1}.
\end{equation}
An obvious realisation of this isometry is the map \( z \mapsto -z \): hence we require
\begin{equation}
  \label{eq:10}
  f(z) = f(r.z) = f(-z),
\end{equation}
i.e. that \( f \) is even. Since \( \wp(z) \) is an even function, the decomposition (\ref{eq:8}) gives that \( R \circ \wp \) is an even function and \( \wp' \cdot S \circ \wp \) is odd, and hence the meromorphic p2-invariant functions are those of the form
\begin{equation}
  \label{eq:11}
  f(z) = R(\wp(z))
\end{equation}
for \( R \) a rational function.

\subsection{p4}
\label{sec:p4}

This group has a square lattice cell; hence we are forced into the \emph{lemniscate case}, where the period ratio is \( \omega_2/\omega_1 = i \), say. It is the group generated by the translations, and a rotation by \( \pi/2 \), \(r^4=1\). An obvious representation of this rotation is multiplication by \(i\). Since restricting to the subgroup generated by \(a,b\) and \( r^2 \) gives a p2, we already know that p4-invariant functions have to be even, and hence they are certainly of the form \( f(z) = R(\wp(z)) \). Therefore, we now have to understand what \(r\) does to \(\wp\).

\begin{lemma}
  Under the action of \(r\), we have
  \begin{equation}
    \label{eq:12}
    \wp(r.z) = -\wp(z).
  \end{equation}
\end{lemma}
\begin{proof}
  There are a number of ways of proving this. One possible way is to consider the series definition, or the expression of \(\wp\) as a power series used in deriving the differential equation. We shall do it using our decomposition: \( \wp(iz) = R_i (\wp(z))+S_i (\wp(z))\wp'(z) \). It is apparent by considering the principal part of the Laurent series at zero, \( z^{-2} \), that \( \wp \) does not map to itself under \(z \mapsto iz\). Clearly since \( \wp(z) \) is even, \( \wp(iz) \) is also even. Now, because \( \omega_2=i\omega_1 \), \( \wp(iz) \) has the same poles as \( \wp(z) \), with the same orders as those of \(\wp(z) \). Therefore \( \wp(iz) = R_i (\wp(z)) \). Moreover, we have
  \begin{equation}
    \label{eq:13}
    \wp(z)=\wp(-z) = \wp(i(iz)) = R_i(\wp(iz)) = R_i(R_i(\wp(z))),
  \end{equation}
  so \( R_i \) is a rational function with \( R_i(R_i(w))=w \) for any complex \(w\). Since the poles are preserved, \( R_i: \infty \mapsto \infty \), and hence \(R_i\) is an automorphism of \(\complexes\): \( R_i(w) = aw+b \). The only functions of this form that have order 2 are those with \(a=-1\). \( b=0\) because the constant term in the Laurent series expansion of \( \wp(z) \) at \(z=0\) is unaffected by applying \( r \).
\end{proof}

Having obtained this result, it is clear that we need
\begin{equation}
  \label{eq:14}
  f(z) = f(iz)= R(\wp(iz)) = R(-\wp(z)) = R(\wp(z)),
\end{equation}
so \(R\) is an even rational function, i.e. a function of \( R(z)=T(z^2) \) for some rational function \(T\). Then \(f(z) = R(\wp(z)) = T(\wp(z)^2)\) is invariant under the action of the generators of p4, and hence under the whole group action.

\subsection{p3}
\label{sec:p3}

By now the process should be clear: we consider the action of the stabiliser group of a lattice point on the basic functions \(\wp,\wp' \), then obtain the relation that the rational functions \(R,S\) must satisfy.

Here, we have a hexagonal lattice. The group is generated by the translations, and the rotation of order 3, which we shall again call \(r\). This time, \(r\) may be represented by multiplication by \(\omega\), where \(\omega\) is a nontrivial cube root of unity. We show:
\begin{lemma}
  Under the action of \( r \), we have
  \begin{equation}
    \label{eq:15}
    \wp(r.z) = \wp(\omega z) = \omega \wp(z), \quad \wp'(\omega z) = \wp'(z).
  \end{equation}
\end{lemma}
\begin{proof}
  Since the lattice is hexagonal, the given transformation maps poles to poles; by the same argument as p2, we conclude that the automorphism of the Riemann sphere which \(r\) induces descends to an automorphism of \( \complexes \) that cubes to the identity: in particular, they are affine transformations. We now check this automorphism is the one given in the Lemma.
  
  Clearly \( \wp(z)-z^{-2} \) and \( \wp'(z)+2z^{-3} \) are analytic at \(z=0\); since this must clearly remain true if we just replace \( z \) by \(\omega z\), the given multiplications are the only ones possible. Considering the constant terms in the Laurent expansions at zero implies that the affine transformation can only be a rotation, and hence we obtain the result.
\end{proof}

Now, of course, we just need to decide what sort of rational functions are invariant under \emph{these} transformations. We obviously have this invariance if and only if \(R(w),S(w)\) are invariant under \( w \mapsto \omega w \). What sort of functions are these? Clearly if \( w=\alpha \) is a root of the numerator (or the denominator), so are \( \omega\alpha\) and \( \omega^2\alpha \). Hence the polynomial contains the factor
\begin{equation*}
  (w-\alpha)(w-\omega\alpha)(w-\omega^2\alpha) = w^3 - (1+\omega+\omega^2)\alpha w^2 + (1+\omega^2+\omega^4)\alpha^2 w -\omega^3\alpha^3 = w^3-\alpha^3,
\end{equation*}
since the roots of unity sum to zero. Hence \( R,S \) are rational functions in \(w^3\).

\begin{remark}
  In fact, this argument allows us to deduce that \( g_2=0 \) for the hexagonal lattice (this is then a scaling of the \emph{equianharmonic case}): otherwise, the equation
  \begin{equation}
    \label{eq:16}
    \wp'^2 = 4\wp^3-g_2 \wp -g_3
  \end{equation}
  would be inconsistent on sending \( z \mapsto \omega z\), since the first two terms are invariant, but the third is not. A similar argument applies to the previous case and shows that \( g_3=0 \) for a square lattice.
\end{remark}

Given that \( g_2=0\), there is a simpler way to express this: since \( \wp^3=\wp'^2+g_3 \), we immediately have

\begin{corollary}
  The invariant meromorphic functions for p3 are precisely rational functions of \( \wp'(z) \), \( R(\wp'(z)) \).
\end{corollary}
\begin{proof}
  This follows immediately from the above remark and the lemma.
\end{proof}

\begin{remark}
  For the case of a hexagonal lattice, there is an alternative presentation available in terms of \emph{Dixon's elliptic function} \( \operatorname{cm}(z) \), which may be given by
  \begin{equation}
    \label{eq:18}
    \operatorname{cm}(z) = 1+\frac{2}{3\wp'(z)-1};
  \end{equation}
  it is clear from this expression that any rational function of \(\wp'\) is also a rational function of this function. Sadly this does not extend to any other cases.
\end{remark}

\subsection{p6}
\label{sec:p6}

Fortunately, we do not need to carry out any further calculation here: p6 is generated by the translations, and two commuting rotations: \(r\), of order 2, and \(s\), of order 3. But we know what elliptic functions are invariant under each of these separately: those invariant under both are then just the intersection of the two sets. We obtain

\begin{proposition}
  The meromorphic p6-invariant functions are precisely the functions
  \begin{equation}
    \label{eq:17}
    f(z) = U(\wp(z)^3) = V(\wp'(z)^2),
  \end{equation}
  where \(U,V\) are rational functions.
\end{proposition}
\begin{proof}
  Since \( f \) is invariant under \( r \), the argument of \autoref{sec:p2} shows that it must be even. Hence \( S\equiv 0 \) in the usual decomposition. The argument for p3 in the previous subsection shows that \( R(w) \) is a function of \( w^3 \), so we can take instead \( U(y) = R(w) \) with \( y=w^3 \), which gives the first equality; the relationship \( \wp^3=\wp'^2-g_3 \) gives the second.
\end{proof}

\subsection{Summary}
\label{sec:summary}

We have now classified all meromorphic functions invariant under the action of the sensible groups; we give here a summary of the results obtained in this section.

\begin{table}[ht]
  \centering
  \begin{tabular}{c|l}
    \(G\) & form of \( G \)-invariant functions \\
    \hline
    p1 & \( R(\wp) + S(\wp) \cdot \wp' \) \\
    p2 &  \( R(\wp) \)\\
    p4 & \( R(\wp^2) \) \\
    p3 &  \( R(\wp') \) \\
    p6 & \( R(\wp'^2) \) or \( S(\wp^3) \)
  \end{tabular}
  \caption{\(G\)-invariant meromorphic functions for the sensible wallpaper groups; \(R,S\) are arbitrary rational functions.}
  \label{tab:sensibleinvariants}
\end{table}

\section{Generalisations to groups containing reflexions}
\label{sec:gener-refl-groups}

Because a meromorphic function of \( \bar{z}\) is no longer meromorphic, there are no nonconstant meromorphic functions invariant under \( m.z=\bar{z} \). It is therefore necessary to consider an alternative group action to continue producing invariant meromorphic functions. One way to do this is to modify the group action in the case of reflexions, and since we are discussing meromorphic functions, a natural idea is to associate sense-reversing transformations in the domain with sense-reversing transformations in the codomain, i.e., take conjugates in both domain and codomain; this will preserve conformality and hence meromorphicity.\footnote{We discuss some alternatives to this in the last section, but this seems to be the simplest extension.}

Concretely, we replace the action of a general (glide-)reflexion \(m\), with \( m.z = a\bar{z}+b \) for \( \abs{a}=1 \), by
\begin{equation}
  \label{eq:39}
  \bar{m}.f(z) = \overline{f(m.z)} = \overline{f(a\bar{z}+b)},
\end{equation}
which acts on both sides of \(f\), and retains the meromorphicity of \( f \) as desired.

There are essentially three different cases to consider: in increasing order of complexity, we have:
\begin{enumerate}
\item Eight groups that can be generated by adding one reflexion in the real axis to the sensible groups, as follows:
  \begin{enumerate}
  \item pm, p2mm, p4mm, p31m, and p6mm can be obtained from the sensible groups by taking one translation parallel to the real axis, and adding reflexion in the real axis (a ``rectiform basis'').
  \item cm, c2mm and p3m1 can be obtained from the sensible groups p1, p2 and p3 by choosing the translations to be interchanged by reflexion in the real axis (a ``rhombic basis'').
  \end{enumerate}
\item Two groups which requires a non-axial reflexion, p2mg and p4mg.
\item Two groups which require a glide reflexion, pg and p2gg.
\end{enumerate}

The cases in 1 are much easier to understand than those in 2 and 3.

\subsection{Groups containing axial reflexions}
\label{sec:groups-cont-refl}

We begin with a discussion of general reflexions. The following lemma is standard:

\begin{lemma}
  Let \(m\) act by complex conjugation, \( m.z=\bar{z} \). Then \(\bar{m}.f(z) = f(z)\) iff \(f(z)\) is real for real \(z\).
\end{lemma}
which we recognise as the Schwarz reflexion principle; we recall that the proof examines the meromorphic function \( f(z)-\bar{m}.f(z) \) on the real axis, where it is precisely equal to \( \Im(f(z)) \).

\begin{corollary}
\label{thm:mirrorcor}
  Let \( m \) be a reflexion in a line \( L \subset \complexes \). Then \(\bar{m}.f(z) = f(z)\) iff \(f(z)\) is real for \(z \in L \).
\end{corollary}

This has essentially the same proof as the previous lemma.

\begin{corollary}
  The Weierstrass functions \(\wp(z),\wp'(z)\) are real for real \(z\) iff their lattices are invariant under conjugation.
\end{corollary}
\begin{proof}
  This is clear: if not, the function \( f(z)-\bar{m}.f(z) \) from the previous lemma has poles that are not cancelled out. If so, \( f(z)-\bar{m}.f(z) \) has no poles and is elliptic; hence it is constant, and examining the Laurent series at \(z=0\) shows that it is zero.
\end{proof}

The next proposition gives the characterisation we want.

\begin{proposition}
  Let \(W\) be a wallpaper group containing reflexions, and define \( \bar{W} \) as the group obtained by replacing \( m \mapsto \bar{m} \) for any reflexion \(m \in W\) (it is easy to check that this is consistent). Then \(f\) is an elliptic function invariant under \(\bar{W}\) if and only if it is
  \begin{enumerate}
  \item invariant under the sensible subgroup of \(W\), and
  \item real on the lines invariant under \(m\), for every \(m \in W\).
  \end{enumerate}
\end{proposition}

\begin{proof}
  Every transformation in \( \bar{W} \) can be written as a composition of sensible maps and reflexions: hence \( f \) is invariant if and only if it is invariant under each of these separately. Invariance under sensible transformations is the first condition, and by \autoref{thm:mirrorcor}, the second condition occurs if and only if \( f \) is invariant under \(\bar{m}\).
\end{proof}

As we noted above, if we align the lattice correctly, for most of the wallpaper groups containing reflexions it is sufficient to check the real line. In particular, this is the case for pm, p2mm, p4mm, p31m, and p6mm with rectiform lattices, and cm, c2mm, and p3m1 with rhombic lattices. In these cases, all other reflexions are generated by combining one in the real axis with elements of the sensible subgroup.

There is one exception to this, and it is the group p2mg, which uses the reflexion
\begin{equation}
  \label{eq:36}
  R.z=\bar{z}+\tfrac{1}{2}\beta i,
\end{equation}
which leaves invariant the ``quarter-translation line'' \(\Im(z)=\tfrac{1}{4}\beta\). Since this condition is rather hard to visualise, we shall exploit a different technique to produce a more comprehensible set of functions.

\subsection{pg and p2mg}
\label{sec:pg-p2mg}

From now on, the discussion becomes considerably more technical. Our lattices are all at worst rectangular, with \( \omega_1=\alpha \) and \(\omega_2=i\beta\) for real \(\alpha\) and \(\beta\).

We have the following result, from the addition formula:
\begin{equation}
  \label{eq:22}
  \wp(z+\tfrac{1}{2}\omega_i) = e_i + \frac{(e_i-e_j)(e_i-e_k)}{\wp(z)-e_i};
\end{equation}
differentiating this gives
\begin{equation}
  \label{eq:23}
  \wp'(z+\tfrac{1}{2}\omega_i) = \frac{(e_i-e_j)(e_i-e_k)}{(\wp(z)-e_i)^2}\wp'(z)
\end{equation}

Therefore, to understand the glide reflexion \( L.z = \bar{z}+\tfrac{1}{2}\alpha \), used in pg and p4mg, we need to find the rational functions such that
\begin{equation}
  \label{eq:24}
  R(u) = \overline{R\left( a+\frac{b}{u-a} \right)}, \quad S(u) = \frac{b}{(u-a)^2} \overline{S\left(a+\frac{b}{u-a}\right)},
\end{equation}
for fixed real constants \(a,b\) with \(b \neq 0\). In fact, we have the inequalities\footnote{See, for example, \cite{Walker:1996cr} §~8.4.2}
\begin{equation}
  \label{eq:26}
  e_1 > e_3 > e_2
\end{equation}
for a rectangular lattice, so whichever of \( e_1 \) or \(e_2\) we choose, we shall end up with \(b>0\), and hence we can write \( b=c^2 \) for a unique positive \(c\).

A sensible thing to do here is change to functions of \(w=(u-a)/c\), which gives the equations
\begin{equation}
  \label{eq:25}
  Q(w)=\overline{Q\left( \frac{1}{w} \right)}, \quad T(w) = \frac{1}{w^2} \overline{T\left( \frac{1}{w}\right)},
\end{equation}
where \( Q(w)=R(c(w+a))\) and \( T(w)=S(c(w+a)) \). It is easy to show the following:

\begin{proposition}
\label{thm:conjinvinvarratfns}
  The functions \(Q\) described above are all of the form
  \begin{equation}
    \label{eq:27}
    Q(w) = C w^p \frac{\prod_k \left((w-a_k)(\bar{a}_kw-1)\right)^{\lambda_k}}{\prod_l \left((w-b_l)(\bar{b}_lw-1)\right)^{\mu_l}},
  \end{equation}
  where \(a_k,b_l \neq 0\), \( \sum_l \mu_l = p+\sum_k \lambda_k \), \(C\) is real, and any \(T(w)\) is given by multiplying a function satisfying the same conditions as \(Q(w)\) by \(w\).
\end{proposition}

\begin{proof}
  Suppose \(a \neq 0\) and \( Q(a)=0 \). Then we must have \( \bar{Q}(1/a)=0 \), and hence \( \bar{a}^{-1} \) is also a root of \(Q\). Thus \(Q\) has a factor of \( (w-a)(\bar{a}w-1) \) in its numerator. We can repeat this construction to eliminate any zeros and poles of \(Q\) away from the origin, leaving us with \( Cw^p \). We now have to obtain the conditions on \(C\) and the sums. To find these, we just calculate \( \overline{Q(1/w)} \) explicitly, and find that \( \bar{C}=C\) and we acquire an extra factor of \( w^{-p-2\sum_k \lambda_k+2\sum_l \mu_l} \); hence \(C\) is real and
  \begin{equation}
    \label{eq:28}
    \sum_l \mu_l = p+\sum_k \lambda_k,
  \end{equation}
  as required. For the second part, observe that \( Q(w)=wT(w)\) satisfies the conditions of the first part of the proposition.
\end{proof}

Therefore the \(\overline{\text{pg}}\)-invariant functions are precisely the functions
\begin{equation}
  \label{eq:29}
  Q((\wp-a)/c)+T((\wp-a)/c)\wp',
\end{equation}
where \(Q,T\) are as in \autoref{thm:conjinvinvarratfns}, and \(a=e_1\), \(c=\sqrt{(e_1-e_2)(e_1-e_3)}\).

\(\overline{\text{p2mg}}\) can be presented as \(\overline{\text{pg}}\) with a 2-fold rotation, so the \(\overline{\text{p2mg}}\)-invariant functions are just the \(\overline{\text{pg}}\)-invariant ones with \( T \equiv 0 \).

Alternatively, \(\overline{\text{p2mg}}\) is similar to the usual reflective case in p2mm, but the reflexion is given by \( m.z=\bar{z}+\tfrac{1}{2}\beta i \). Now we know about glide reflexions, we can understand this reflexion in much the same way: in particular, we have
\begin{equation}
  \label{eq:35}
  \bar{R}.\wp(z) = \overline{\wp\left(\overline{z-\tfrac{1}{2}\beta i}\right)} = \wp(z-\tfrac{1}{2}\beta i) = e_2 + \frac{(e_2-e_1)(e_2-e_3)}{\wp(z)-e_2},
\end{equation}
and we already know all about this transformation. Since p2mg contains a 2-fold rotation, \( T \equiv 0 \), and we have that the \(\overline{\text{p2mg}}\)-invariant functions are the \(Q(w)\) from \autoref{thm:conjinvinvarratfns}, but this time with
\begin{equation}
  \label{eq:37}
  w = \frac{\wp(z)-e_2}{\sqrt{(e_2-e_1)(e_2-e_3)}};
\end{equation}
the two representations are related by exchanging \(e_1\) and \(e_2\), since the glide reflexion axes are perpendicular to the mirror axes.

The final two groups are

\subsection{p2gg and p4mg}
\label{sec:p2gg-p4mg}

This is similar to the last case in complexity, but thankfully we only have to understand the behaviour of the functions \( R(\wp) \) since the group contains a 2-fold rotation, and so has p2 as a subgroup. The required glide reflexion is given by
\begin{equation}
  \label{eq:30}
  L.z = \overline{z+\tfrac{1}{2}(\alpha+\beta i)},
\end{equation}
which is, naturally, not the same as the one we discussed in the previous subsection. We have
\begin{equation}
  \label{eq:31}
  \overline{\wp(L.z)} = \wp(z+\tfrac{1}{2}(\alpha+\beta i)) = e_3 + \frac{(e_3-e_2)(e_3-e_1)}{\wp(z)-e_3},
\end{equation}
which is of the form
\begin{equation}
  \label{eq:32}
  z \mapsto a-\frac{c^2}{z-a}.
\end{equation}
It is fairly clear what we should do now: find the rational functions of \(z\) invariant under this transformation. Once again the sensible thing to do is set \(w=(z-a)/c\); and now we need to establish conditions on \(Q\) with
\begin{equation}
  \label{eq:33}
  Q(w) = \overline{Q\left( -\frac{1}{w} \right)}.
\end{equation}
The calculation proceeds exactly as before, with some additional minus signs, and we find
\begin{proposition}
\label{thm:p2gginvars}
  The rational functions \(Q\) in the preceding paragraph are precisely those of the form
  \begin{equation}
    \label{eq:34}
    Q(w) = Cw^p \frac{\prod_k \left((w-a_k)(\bar{a}_kw+1)\right)^{\lambda_k}}{\prod_l \left((w-b_l)(\bar{b}_lw+1)\right)^{\mu_l}},
  \end{equation}
  with the same conditions as in \autoref{thm:conjinvinvarratfns}.
\end{proposition}
Then the \(\overline{\text{p2gg}}\)-invariant functions are exactly
\begin{equation}
\label{eq:36}
  Q((\wp-a)/c),
\end{equation}
with \(Q\) as in the above proposition, and \( a=e_3 \), \( c=\sqrt{(e_1-e_3)(e_3-e_2)} \).

A subset of these are \(\overline{\text{p4mg}}\)-invariant: the only additional required condition is that the lattice is square and that the rational functions are even, as we showed was required for p4-invariance, so we automatically restrict to \( Q((w-a)/c)=P(w^2) \) as before.

\subsection{Summary}
\label{sec:summary-1}

We summarise the invariant functions for the groups containing reflexions and glide reflexions here.

\begin{table}[ht]
  \centering
  \begin{tabular}{c|c|l|l}
    \(G\) & Lattice & \(G\)-invariant functions & Conditions \\
    \hline
    pm & \multirow{2}{*}{\( \alpha,\beta i \)} & \( R(\wp)+S(\wp) \cdot \wp' \) & \multirow{8}{*}{\(R(w),S(w)\) real for real \(w\)} \\
    p2mm && \( R(\wp) \) & \\
    \cline{1-3}
    p4mm & \( \alpha,\alpha i \) & \(R(\wp^2)\) & \\
    \cline{1-3}
    p3m1 & \multirow{2}{*}{\( \alpha,\alpha \omega  \)} & \(R(\wp')\) & \\
    p6mm & & \(R(\wp'^2)\) or \(S(\wp^3)\) & \\
    \cline{1-3}
    cm &\multirow{2}{*}{\( \alpha \pm \beta i \)} & \( R(\wp)+S(\wp) \cdot \wp' \) & \\
    c2mm && \( R(\wp) \) & \\
    \cline{1-3}
    p3m1 & \( \alpha \pm \alpha\sqrt{3}i \) & \(R(\wp')\) & \\
    \hline
    pg &\multirow{3}{*}{\( \alpha,\beta i \)} & \(Q((\wp-a)/c)+T((\wp-a)/c) \cdot \wp' \) & \multirow{2}{*}{\(Q,T,a,c\) as in \autoref{thm:conjinvinvarratfns}}  \\
    p2mg && \(Q((\wp-a)/c)\) & \\
    \cline{1-1}\cline{3-4}
    p2gg && \( Q((\wp-a)/c) \) & \multirow{2}{*}{\(Q,a,c\) as in \autoref{thm:p2gginvars}} \\
    \cline{2-2}
    p4mg &\( \alpha,\alpha i \) & \( Q((\wp-a)/c) = P(\wp^2) \) &
  \end{tabular}
  \caption{Meromorphic functions invariant under the modified action of the wallpaper groups containing reflexions: \( R,S,Q,T,P \) are all rational functions.}
  \label{tab:modifiedreflexioninvariants}
\end{table}

\section{Generalisations and extensions}
\label{sec:generalisation}

The action of all the groups considered respects the multiplicative and additive structure in the field of meromorphic functions \( \torus_{\tau} \to \mathbb{CP}^1 \), so all the restrictions we have produced give subfields of this field; the author must plead ignorance of any applications of this, however.

\subsection{Other geometries}
\label{sec:other-geometries}

A similar, but rather simpler idea is to consider the meromorphic functions invariant under the \emph{frieze groups}, which contain only one translation: they will form subsets of the space of trigonometrical series, although the noncompactness of the domain as a quotient of \( \complexes \) (\textit{viz.} the cylinder) renders the theory rather less restrictive. One could consider imposing meromorphicity at the cusp, as with automorphic functions, to lessen this issue.\footnote{The basics of this have been discussed in the book \cite{Farris:2015qf} of Farris, although I believe that the general theory is as yet unrealised.}

Two other possibilities suggest themselvs: discrete symmetries of the sphere, via stereographic projection, and the hyperbolic plane, using the upper half-plane. For example, it is well-known that
\begin{theorem}
  A discrete sensible symmetry group of the sphere is isomorphic to one of:
  \begin{enumerate}
  \item The cyclic group \(C_n\) for some \(n\geq 1\),
  \item The dihedral group \(D_n\) for some \( n \geq 1 \),
  \item The rotational symmetries of the regular tetrahedron \(T\), octahedron/cube \(O\), or icosahedron/dodecahedron \(I\),
  \end{enumerate}
  and all of these act solely by rotations.
\end{theorem}

If we include reflexions, the cyclic family splits into four infinite families, and the dihedral into three; these correspond to finite versions of each of the frieze groups. The polyhedral groups increase to seven: one extra nonsensible group for each, with an extra for the tetrahedron since there are two ways of doubling the tetrahedral group by adding a reflexion.

These groups act on the complex plane as finite subgroups of the M\"obius group, and complex conjugation for reflexion. Of course, all of these groups will have their own invariant functions, given by meromorphic functions \( \mathbb{CP}^1 \to \mathbb{CP}^1 \), i.e. rational functions. Farris also discusses a subset of these: since the group is finite, one can average over orbits of a point to produce a set of invariant functions. Again, there does not seem to have been a general classification produced, probably due to the complexity of some of the M\"obius transformations involved.

Naturally, the other possibility is the hyperbolic plane. Here the list of types of symmetry groups is much longer (indeed, it is infinite); this is intimately related to the large number of tilings of the hyperbolic plane, each with its own symmetry groups. As with the elliptic case, the basic theory is venerable, including the Fuchsian groups and their invariant functions; more recent work has considered the two-dimensional hyperbolic orbifolds, which include both the reflexions, and considerably greater complexity.

\subsection{Other possible extensions to the non-sensible groups}
\label{sec:other-poss-extens}

In the current answer to \cite{36737}, it is suggested that we consider real parts of an elliptic functions that are \(G\)-invariant for each wallpaper group \(G\). This strikes us as a less elegant extension, since taking the real part is not compatible with the complex multiplication, and so it appears that we lose the field structure on the meromorphic functions.\footnote{A more complicated argument shows that \(\Re(f \cdot g)\) is invariant if and only if \( \Re(f) \) and \(\Re(g)\) are, but one also requires understanding the imaginary part, perhaps contrary to the spirit of this idea.}

It also happens that such functions are precisely the real parts of the functions obeying our modified reflexion law \( \bar{m}.f(z) = \overline{f(m.z)} \): note that \( \Re(\overline{w}) = \Re(w) \), so certainly every \( \bar{G}\)-invariant function has a \(G\)-invariant real part. Conversely, suppose \(f\) is meromorphic and \( \Re(f) \) is \(G\)-invariant. Then the Schwarz reflexion principle implies that \( f(m.z) = \overline{f(z)} \), so \(f\) is invariant under the action of \( \bar{m} \). (Another way to look at it is that \(f(z)\) is conformal, so so is \(f(m.z)\) since \(m\) is a conformal transformation: it must be anticonformal, and the only way for this to happen if \( \Re(f) \) is invariant is to change \( \Im(f) \) to \(-\Im(f)\).)

\subsection{Symmetries of the graphs and equivariance}
\label{sec:symmetries-graphs}

There are two other generalisations that one may consider: suppose that \( G \) is a wallpaper group, \( H \) a group of symmetries of \( \mathbb{CP}^1 \), and \( \phi:G \to H \) is a homomorphism. Then \(G\) acts on \(\complexes\), and \(H\) acts on \( \mathbb{CP}^1 \), so they can be combined to give an action which we continue to call \( \bar{G} \), more general than what we have previously considered, on the space of meromorphic functions \( \complexes \to \mathbb{CP}^1 \):
\begin{equation}
  \label{eq:40}
  \bar{g}.(f(z)) = (\phi(g).f)(g.z) = (\phi(g) \circ f) (g.z).
\end{equation}
Composition is given by
\begin{align*}
  \bar{g}_1.(\bar{g}_2.(f(z))) &= \bar{g}_1.((\phi(g_2).f)(g_2.z)) \\
                               &= (\phi(g_1) .(\phi(g_2) .f)) (g_1.(g_2.z)) \\
                               &= (\phi(g_1g_2).f) ((g_1g_2).z) \\
                               &= \overline{g_1g_2}.(f(z)),
\end{align*}
where the third inequality uses the \(G\)- and \(H\)-actions, and the others are the definition of the \(\bar{G}\) action.\footnote{This can be regarded as a quotient of the action of \( (G,\phi(G)) \) on \( X \times Y \) by the relation \( y=f(x) \).} 

For example, the idea we discussed in the previous section is a special case of this when \(H\) is the group generated by complex conjugation \( j \), and \( \phi \) sends sense-preserving transformations to the identity, sense-reversing to \(j\). As an example where \(H\) is an isometry of \(\complexes\), we could consider the action of a lattice group sending \( f(z) \) to \( f(z+m\omega_1+n\omega_2)-2m \eta_1-2n \eta_2 \), for \(m,n \in \integers^2\), where \( \eta_i \) are the eta constants: one function invariant under this transformation is the Weierstrass \(\zeta\)-function. Yet another example is provided by the the ``colour-turning wallpaper functions'' from \cite{Farris:2002dq}, which satisfy in our notation
\begin{equation*}
  f(k.x) = \psi(k) f(x),
\end{equation*}
where \( \psi: G \to H \), and \(H\) is a subgroup of the multiplicative group of roots of unity.

It is apparent from stereographic projection that \(H\) may be a spherical symmetry group, (although the second example above showing that \(H\) need not be) but other restrictions appear more subtle. Of course, the orders of the elements of \( \phi(G) \) are divisible by those of \(G\): in particular, there are no elements of order \(5\), so we can ignore the icosahedral group, and other groups with rotations of order \(5\).

\begin{remark}
  This viewpoint also suggests that we should instead be considering the symmetries of the graph of \(f\) as a holomorphic section of the projective line bundle \( \complexes \times \mathbb{CP}^1 \), or, if we assume the \(\phi\)-image of the translation subgroup is the identity, the projective line bundle over the torus, \( \torus_{\tau} \times \mathbb{CP}^1 \), where \( \tau \in \mathbb{H} \) is the period ratio. Pursuing this even further would lead to considering the symmetries of maps from one Riemann surface to another: given \( \bar{G}=\{(g, \phi(g)) : g \in G \} \) acting on Riemann surfaces \( X \times Y \) by \( \bar{g}.(x,y) = (g.x,\phi(g).y) \) what are the holomorphic \( f:X \to Y \) with \( \bar{g}.(z,f(z)) = (g.z,\phi(g).(f(g.z)))) = (g.z,f(z)) \)?
\end{remark}

\section*{Acknowledgements}
\label{sec:acknowledgements}

I would like to thank Imre Leader and Jack Button for their advice in preparing this paper.

\printbibliography

\end{document}